\documentclass[11pt]{amsart}
\usepackage{amsmath,amsfonts,amsthm,amssymb,amscd}

\def\classification#1{\def\@class{#1}}
\classification{\null}

\DeclareMathAlphabet{\mathscr}{OT1}{rsfs}{n}{it}

\newcommand{\R}{{\mathbb R}}

\newcommand{\C}{\mathbb{C}}
\newcommand{\F}{\mathbb{F}}

\newtheorem{theorem}{Theorem}
\newtheorem{lemma}[theorem]{Lemma}
\newtheorem{corollary}[theorem]{Corollary}
\theoremstyle{remark}
\newtheorem{remark}[theorem]{Remark}

\title{On new sum-product type estimates}
\author{Sergei V. Konyagin}
\address{ Sergei V. Konyagin, Steklov Mathematical Institute, 8 Gubkin Street, Moscow 119991, Russia}
\email{konyagin@mi.ras.ru}
\author{Misha Rudnev}
\address{Misha Rudnev, Department of Mathematics, University of Bristol, Bristol BS8 1TW, United Kingdom}
\email{m.rudnev@bristol.ac.uk}

\subjclass[2000]{68R05, 11B75}
\begin{document}
\begin{abstract} New lower bounds involving sum, difference, product, and ratio
sets of a set $A\subset \C$ are given. The estimates involving the sum set
match, up to constants, the state-of-the-art estimates, proven by Solymosi for
the reals, and are obtained by
generalising his approach to the complex plane. The bounds involving the
difference set improve the currently best known ones, also due to Solymosi, in both the real and complex cases
by means of combining
the Szemer\'edi-Trotter theorem with an arithmetic combinatorics technique.
\end{abstract}

\maketitle

\section{Introduction} Erd\H os and Szemer\'edi, \cite{ES},  conjectured  that if $A$ is a finite set of integers, then for any $\varepsilon>0$, as the cardinality $|A|\rightarrow\infty$,
$$
|A+A|+|A\cdot A|\geq |A|^{2-\varepsilon}.
$$
Above,
$$
A+A = \{a_1+a_2:\,a_1,a_2\in A\}
$$
is called the sum set of $A$, the product $A\cdot A$, difference $A-A$, and ratio $A:A$ sets being similarly defined.
(In the latter case one should not divide by zero.)

Variations of the Erd\H os-Szemer\'edi conjecture address subsets of other
rings or fields -- see \cite{TT} for a general discussion and \cite{BJ} for a
new quantitative sum-product estimate in function fields -- as well as
replacing, e.g., the sum set with the difference set $A-A.$ The conjecture is far from being
settled, and therefore current ``world records" vary
with such variations of the problem.

The best result for $A\subset \R$, for instance, is due to Solymosi (\cite{So1}), claiming
\begin{equation}\label{wr}
|A+A|+|A\cdot A|\gg \frac{|A|^{1+\frac{1}{3}}}{\log^{\frac{1}{3}} |A|},
\end{equation}
and without the logarithmic term if $A\cdot A$ is replaced by $A:A$.  The notation
$\ll,\;\gg$ is being used throughout to suppress absolute
constants in inequalities, that is constants which do not depend on the
parameter $|A|$.

At the first glance, the construction in \cite{So1} appears to be specific for
reals, nor does  it seem to allow for replacing the sum set $A+A$ with the difference set $A-A$.
So, if $A\subset \C$ or if $A+A$ for reals gets replaced by $A-A$, the best known
result comes from an older paper of Solymosi \cite{So}, claiming
\begin{equation}\label{wr1}
|A- A|+|A\cdot A|\gg \frac{|A|^{1+\frac{3}{11}}}{\log^{\frac{3}{11}} |A|},
\end{equation}
and without the logarithmic term if $A\cdot A$ gets replaced by $A:A$.

In this paper we show, firstly, that the order-based observation which allowed Solymosi to prove (\ref{wr}), namely the fact that for real positive $a,b,c,d$
 $$\left(\frac{a}{b}< \frac{c}{d}\right)\qquad \Rightarrow \qquad\left( \frac{a}{b}<\frac{a+c}{b+d}<\frac{c}{d}\right)$$
 admits a natural extension to the complex case. We therefore extend
the estimate (\ref{wr}) to the case $A\subset \C$. This is the content of the forthcoming Theorem \ref{sumest}.

Secondly, we prove new
estimates involving the difference set, for $A\subset \C$, which improve on
(\ref{wr1}). For this we use rather
different arguments, relying on the Szemer\'edi-Trotter theorem,
combined with an arithmetic technique. This is the content of the forthcoming Theorem \ref{diffest}.

We remark that both Theorem \ref{sumest} and Theorem \ref{diffest}, even though they apply to the case $A\subset \C$, rely crucially on the metric properties of the Euclidean space, and we presently do not see how the ideas behind them could apply to the
case when $A$ is a small subset of a prime residue field $\mathbb Z_p$ of large characteristic, where the best known exponent in the sum-product inequality is $\frac{12}{11}$, up to a logarithmic factor in $|A|$, see \cite{MR}.

\medskip
We now formulate our main results.

\begin{theorem}\label{sumest}
For any finite $A\subset \C$, with at least two elements, one has
the following estimates:
\begin{equation}\begin{aligned}
|A+A| + |A:A| &\gg & |A|^{1+\frac{1}{3}}, \\
|A+A| + |A\cdot A| &\gg & \frac{|A|^{1+\frac{1}{3}}}{\log^{\frac{1}{3}} |A|}. \\
\end{aligned}
\label{ressum}\end{equation}
\end{theorem}

\begin{theorem}\label{diffest}
For any finite $A\subset \C$, with at least two elements, one has the following estimates:
\begin{equation}\begin{aligned}
|A-A| + |A:A| &\gg & \frac{|A|^{1+\frac{9}{31}}}{\log^{\frac{4}{31}} |A|}, \\
|A-A| + |A\cdot A| &\gg &\frac{|A|^{1+\frac{11}{39}}}{\log^{\frac{5}{13}} |A|}.\\
\end{aligned}
\label{resdif}\end{equation}
\end{theorem}

\section{Preliminary set-up} \label{prelim}
In this section we develop the preliminary set-up and notation to be used in the forthcoming proofs of Theorems \ref{sumest} and \ref{diffest}.

Since we do not pursue best possible values of the constants, hidden in the
inequalities (\ref{ressum}, \ref{resdif}), we further assume that $0\not\in A$
and $|A|\geq C$, for some absolute constant $C$, which is as large as necessary.

Observe that Theorems \ref{sumest} and \ref{diffest} each claims two different
estimates: one involving the ratio set $A:A$ and the other involving the
product set $A\cdot A$. In order to prove these estimates, we deal with a
certain ``popular'' subset $P$ of the point set $A\times A\subset\C^2$.
Note that if $l\in A: A$ is a ratio, it can be identified with a straight line,
passing through the origin in $\C^2$ and supporting $n(l)$ points of the point
set $A\times A$, where $n(l)$ is the number of realisations of the ratio
$l=\frac{y}{x}:\,x,y\in A.$ As we often refer to ``lines''
throughout the paper, we use the symbol $l$ to denote individual members of the ratio set.

 Even though the proofs of Theorems \ref{sumest} and \ref{diffest} are essentially different, the popular subset $P\subseteq A\times A$ is defined in the same way as to both  theorems. Yet $P$ denotes different point sets apropos of the ratio and product
set cases, which figure within each theorem and are described next. The same concerns the notations $L,N$ pertaining to the point set $P$. In particular, the notation $L$ refers to the set of the corresponding popular ratios, or lines through the origin.

\medskip
{\sf Ratio set case.} In order to establish the estimates involving the ratio set,
the notation $L$ will stand for the set of lines through the origin in $\C^2$,
supporting at least $\frac{1}{2}|A|^2|A:A|^{-1}$ points of $A\times A$ each.
The subset $P$ of $A\times A$ supported on these ``popular'' lines is then such that $|P|\geq  \frac{1}{2}|A|^2$. (Indeed, the lines outside $L$ support at most $\frac{1}{2}|A|^2|A:A|^{-1}\cdot |A:A| =\frac{1}{2}|A|^2$ points.)

The notation $N$ will be used for
the maximum number of points per line in $L$. Trivially, $N\leq |A|,$
and one has $|A|/2\leq |L|\leq |A:A|.$

\medskip
{\sf Product set case.} In order to establish the estimates involving
the product set, the same notations $P,L, N$ will be used for
slightly differently defined, multiplicative energy based quantities.

The multiplicative energy $E_*(A)$ of $A$ is defined as follows:
$$
E_*(A) = |\{ (a_1,\ldots a_4)\in A\times\ldots\times A:\; a_1/a_2 = a_3/a_4\}|.
$$
Since the equation defining $E_*(A)$ can be rearranged as $a_1a_4=a_2a_3$,  by the Cauchy-Schwarz inequality one has
\begin{equation}
E_*(A)\geq \frac{|A|^4}{|A\cdot A|}.
\label{emult}\end{equation}
Geometrically,  $E_*(A)$ is the number of ordered pairs of points of
$A\times A\subset \C^2$, supported on straight lines through the origin, whose slopes $l$
are members of the ratio set $A:A$. A line is identified by its slope $l$ (which is well defined, since $0\not\in A$) and supports some number $n(l)$ points of $A\times A$.

By the pigeonhole principle, there exists some $N\in [1,\ldots,|A|],$  such that if $L$ denotes the set of all lines with $\frac{N}{2} < n(l)\leq N$, then
\begin{equation}
|L|N^2\gg \frac{E_*(A)}{\log|A|} \geq \frac{|A|^4}{|A\cdot A|\log|A|}. \label{min}\end{equation}
(Indeed, it suffices to consider only dyadic values of $N=1,2,\ldots,2^j,\ldots$, with $j=O(\log |A|)$, since trivially $n(l)\leq |A|$.)

Now, in the product set case, let $P$ be a ``popular multiplicative energy''
subset of $A\times A$, containing all points of $A\times A$, supported on the
lines in the above defined set $L$, satisfying (\ref{min}). The quantity $N$
gives the maximum, as well as the approximate number of points of $P$ per line
$l\in L$, that is $|P|\approx |L|N$. (This approximate equality means that
$|P|\ll|L|N$ and $|L|N\ll|P|$.)

\section{Proof of Theorem \ref{sumest}}
Without loss of generality, as we are not pursuing optimal constants in the estimates, we may assume that the set
$A\subset \C\setminus\{0\}$ is located in a reasonably small angular sector,
of angular half-width $|\tan (2\arg z)|<\epsilon$ around the real axis,
with the vertex at $0$, so in particular $0\not\in A+A$. The constant $\epsilon>0$ does not go to zero: it only needs to be small enough for the geometric argument in the end
of the proof of the forthcoming claim to be valid. One can amply set $\epsilon = \frac{1}{100}.$

Theorem \ref{sumest} will follow from the following claim.

\medskip
{\sf Claim.} {\em Let $l_1,l_2$ be two distinct members of the ratio set
$(A:A)\subset \C \cong \R^2$, with some realisations $l_1=\frac{y_1}{x_1}$ and
$l_2= \frac{y_2}{x_2}$, for $x_1,y_1,x_2,y_2\in A$. Consider $l_1,l_2$ as
points in $\R^2$. Then the point $z=\frac{y_1+y_2}{x_1+x_2}$ lies in $\C \cong \R^2$
in some open set $M_{(l_1,l_2)}$, containing  the open straight line interval
$(l_1,l_2)=\{t l_1 +(1-t)l_2,\,t\in(0,1)\}$ and symmetric with respect to this
line interval. Furthermore, consider the ratio set as a vertex set of a tree $T$
in $\R^2$, and let the sum of the Euclidean lengths of the edges of $T$ be
minimum, i.e., let $T$ be a minimum spanning tree on the vertex set $A:A$. Then, if $(l_1,l_2)$ runs over the edges of $T$ (we further write simply $(l_1,l_2)\in T$) the sets $M_{(l_1,l_2)}$ are pairwise disjoint.}

\medskip

The above claim represents a bona fide generalisation of the construction
of Solymosi \cite{So1}, for the positive reals. Here is how the claim applies to the positive real case. The set $A:A$ lies on the positive real axis. The edges of its minimal
spanning tree are consecutive open line intervals
between the vertices, and the sets $M_{(l_1,l_2)}$ are these intervals themselves.

In the forthcoming proof of the claim we will describe the open sets $M_{(l_1,l_2)}$ precisely. Through the rest of this section we assume the claim and show how it results in Theorem \ref{sumest}, by essentially repeating the argument in \cite{So1}.

Indeed, suppose that there are respectively $n(l_1)$ and $n(l_2)$ distinct
representations of some two fixed ratios $l_1,l_2\in A:A$, that is
$l_i= \frac{y_i^{j_i}}{x_i^{j_i}},$ $x_i^{j_i}\in A, y_i^{j_i}\in A$ for $i=1,2$ and $j_i=1,\ldots, n(l_i)$.
From basic linear algebra, the vector sums
 $(x_1^{j_1}+x_2^{j_2},\, y_1^{j_1}+y_2^{j_2})\in \C^2$ attain $n(l_1) n(l_2)$
distinct values for distinct $(j_1,j_2)$. Assuming the claim, on the other hand,
tells one that for all $(j_1,j_2)$, the ratio
$\frac{y_1^{j_1}+y_2^{j_2}}{x_1^{j_1}+x_2^{j_2}}\in \C\cong \R^2$ lies in the set $M_{(l_1,l_2)}.$

Now the fact that the open sets $M_{(l_1,l_2)}$ are pairwise disjoint
 implies that the map
 \begin{equation}\label{map}\begin{array}{c}
 (x_1,y_1)\times (x_2,y_2)\; \rightarrow \; (x_1+x_2,y_1+y_2), \\ \hfill \\ \mbox{for } x_1,y_1,x_2,y_2\in A: \;\exists (l_1,l_2)\in T, \mbox{ with } \frac{y_1}{x_1} = l_1, \frac{y_2}{x_2} = l_2,\end{array}
\end{equation}
 is an injection. Indeed, assuming the contrary suggests that there is a pair of distinct edges, $(l_1,l_2)$ and $(l_1',l_2')$ of the tree $T$, such that
 $(x_1+x_2,y_1+y_2)=(x_1'+x_2',y_1'+y_2')$, where $l_1=\frac{y_1}{x_1},l_2=\frac{y_2}{x_2}, l_1'=\frac{y_1'}{x_1'},l_2'=\frac{y_2'}{x_2'}$. Then, clearly,
 $\frac{y_1+y_2}{x_1+x_2}= \frac{y_1'+y_2'}{x_1'+x_2'}$, which contradicts the claim that $\frac{y_1+y_2}{x_1+x_2}$ and $\frac{y_1'+y_2'}{x_1'+x_2'}$ lie, respectively, in the open sets $M_{(l_1,l_2)}$ and $M_{(l'_1,l'_2)}$, which are pairwise disjoint.

The injectivity of the map (\ref{map}) accounts for the following inequality:

\begin{equation}\label{spe0}
|A+A|^2 \geq \sum_{(l_1,l_2)\in T} n(l_1) n(l_2)
\ge \frac{1}{2}\sum_{(l_1,l_2)\in T}(n(l_1)+n(l_2))\min(n(l_1),n(l_2)).
\end{equation}
The inequality (\ref{spe0}) clearly remains true if one restricts the vertex set of $T$ to any subset of $A:A$, with more than one element, in which case $T$ will be a minimum spanning tree built on these vertices.

It is at this point when one has to distinguish between the ratio and product
set cases by considering as vertices of $T$ only the ratios from the ``popular''
set $L$, defined relative to the ratio or product set case in Section
\ref{prelim}. Given the set of vertices $L$, let $T$ be a minimum spanning tree
built on the vertex set $L$ in $\R^2$. Thus $T$ has $|L|$ vertices and $|L|-1$ edges.

In the ratio set case, one has  $\frac{|A|^2}{2|A:A|}\leq n(l)\leq |A|,\;\forall l\in L$, and thus, from (\ref{spe0}):
\begin{equation}\label{spe1}
|A+A|^2 \geq \frac{|A|^2}{4|A:A|} \sum_{(l_1,l_2)\in T} (n(l_1)+n(l_2))
\geq \frac{|A|^2}{4|A:A|}\sum_{l\in L} n(l) \gg \frac{|A|^4}{|A:A|}.
\end{equation}

In the product set case, where $\frac{N}{2}\leq n(l)\leq N\leq |A|,\;\forall l\in L$,  the claim implies, by (\ref{spe0}) and (\ref{min}), that
\begin{equation}\label{spe2}
|A+A|^2 \geq (|L|-1)\frac{N^2}{4}\gg \frac{E_*(A)}{\log |A|} \geq \frac{|A|^4}{|A\cdot A|\log|A|},
\end{equation}
thus proving the second inequality in (\ref{ressum}). (In view of (\ref{min}) one can assume that $|L|>1$, for otherwise $A\cdot A$ is large enough to ensure (\ref{ressum}) immediately.)

This completes the proof of Theorem \ref{sumest}, conditional on the claim.  \qed

\subsection{Proof of the claim}

Suppose that $x_1,x_2,y_1,y_2\in A$, $\frac{y_1}{x_1}=l_1$, $\frac{y_2}{x_2}=l_2$. Then, with $u=x_2/x_1$, we have

\begin{equation}\label{obs}\frac{y_1+y_2}{x_1+x_2}= \frac{y_1+y_2}{x_1(1+u)} = \frac{l_1}{1+u} + {l_2}\frac{u}{1+u}=l_1+(l_2-l_1)\frac{u}{1+u}.\end{equation}
Since we have assumed that $\tan|2\arg{x_1}|,\tan|2\arg{x_2}| < \epsilon,$
clearly $u$ lies in the open angular wedge $W_\epsilon = \{z:\tan|\arg{z}|<\epsilon\}$ and therefore
$\frac{u}{1+u}$ lies in the image of $W_\epsilon,$
further denoted as $M_\epsilon,$ under the M\"obius map $z'=\frac{z}{1+z}$.

A straightforward calculation shows that $M_\epsilon$ is an open meniscus around
the real line interval $(0,1)$.
The meniscus is formed by the intersection of two open discs
centred respectively at $z_{\pm}=(\frac{1}{2}, \pm \frac{\iota}{2\epsilon})$,
with equal radii $|z_{\pm}|$. It is clearly symmetric around its major axis,
that is the real line interval $(0,1)$. The boundary of each disc
intersects the major axis at the angle, whose tangent equals $\epsilon$, the half-width of $W_\epsilon$.
Clearly, $M_\epsilon$ is amply contained in the open rhombus, whose major diagonal connects
the zero with $1$,  and the minor diagonal has length $\epsilon$.

The meniscus $M_\epsilon$ defines an open set $M_{(l_1,l_2)}$ mentioned in the claim as a composition of a dilation and a translation of $M_\epsilon$: by (\ref{obs}),
$$
\frac{y_1+y_2}{x_1+x_2} \in M_{(l_1,l_2)} = \{ l_1+(l_2-l_1) M_\epsilon\}.
$$
Thus, the set $M_{(l_1,l_2)}$ is contained
in the open rhombus, whose main diagonal is denoted as $e=(l_1,l_2)$, and the minor diagonal has length $\epsilon|l_2-l_1|$. This rhombus will be further denoted as $R_e= R_{(l_1,l_2)}$.

Through the rest of this section, let $L$ be any non-empty subset with more than one element of the ratio set $A:A$. Let $T$ be a minimum spanning tree built on the vertex set $L$. That is $T$ has the minimum net Euclidean length of the edges over all the
 trees with the vertex set $L$.  The tree $T$ has $|L|-1$ edges, which are open straight line segments connecting some pairs of distinct vertices in the set $L$. There are no loops in $T$, and for any pair of distinct vertices
$l_1,l_2\in L$, there is a unique path connecting them.

Through the rest of this section, let us use the uppercase Latin letters $A,B,C,D,\ldots$ for the vertices of $T$, regardless of the rest of the paper, where $A,B,\ldots$ are sets.

First, note the well-known fact that $T$ may not contain intersecting edges. Indeed, suppose
that $(AB)$ and $(CD)$ are edges of $T$ and $(AB)\cap(CD)\neq\emptyset$.
In the tree $T$, there is a unique path from $B$ to $C$ and a unique path from $B$ to $D$. Since $T$ has no loops, one of these two paths, without loss of generality the one from $B$ to $D$, must contain the edge $(CD)$ (for if $(CD)$ is not contained in
either of the two paths, there is a path from $C$ to $D$ other than $(CD)$, via $B$).
Then the path from either $A$ or $B$ to $D$ contains both edges $(AB)$ and $(CD)$. Without loss of generality, let it be the path connecting $A$ and $D$.

Thus, if $(AB)\cap(CD)\neq\emptyset$, these edges can be deleted and replaced by the edges $(AC)$ and $(BD)$,
without violating connectivity or creating loops. On the other hand, $[AC]$ and $[BD]$ are a pair of opposite sides of the convex quadrilateral $ACBD$, while $[AB]$ and $[CD]$ are its diagonals. But the sum of the lengths of either pair of opposite sides
of a convex quadrilateral is smaller than the sum of the lengths of the diagonals. This contradicts the minimality of $T$.

In a minimum spanning tree the angle between adjacent edges is at least
$\frac{\pi}{3}$. To see the latter fact, suppose
that there are two edges $(AB)$ and $(AC)$, with the angle between them at $A$
smaller than $\frac{\pi}{3}$. Then one of the two remaining angles in the
triangle $ABC$ exceeds $\frac{\pi}{3}$ and the edge opposite to it in $T$ can
be deleted and  replaced by the shorter edge $(BC)$,
without violating connectivity or creating loops. This contradicts the
minimality of $T$.

Therefore, the rhombi around adjacent edges cannot intersect, because the tangent of the
half-angle of $R_{(AB)}$ at $A$ or $B$ is just $\epsilon.$
The supposition that the rhombi around a pair of adjacent edges $(AB)$ and $(AC)$
intersect would contradict the fact that the angle between them at $A$ is smaller than $\frac{\pi}{3}$.

Finally, suppose that there is a pair of non-adjacent and non-intersecting
edges $(AB)$ and $(CD)$, such that $R_{(AB)}\cap R_{(CD)}\neq \emptyset.$ Let us show that this also leads to a contradiction if $\epsilon$ is small enough.
The key observation is the following lemma.
\begin{lemma}\label{easylemma}
The vertices $C,D$ cannot lie in the open disk with the diameter $(AB)$.
\end{lemma}
\begin{proof}
Indeed, suppose that, say $C$ lies inside the open disk with the diameter $(AB)$. Then the angle $ACB$ is obtuse. Hence, the edge $(AB)$ can be deleted and replaced in the tree $T$ by one of the shorter line segments $(AC)$ or $(BC)$,  without  violating
connectivity or creating loops.
More precisely, if the unique
path from $A$ to $C$ in $T$ incorporates $(AB)$, then $(AB)$ should be replaced
by $(AC)$, and  otherwise by $(BC)$. This contradicts the minimality of $T$. \end{proof}

Let us use Lemma \ref{easylemma} together with the fact that
$(AB)\cap(CD)=\emptyset$ for the proof of the following lemma.

\begin{lemma}\label{lesseasylemma}
If $R_{(AB)}\cap R_{(CD)}\neq \emptyset$ and $\alpha$ is the angle
between $(AB)$ and $(CD)$ then $\tan\alpha \leq \frac{2\epsilon}{1-\epsilon^2}$.
\end{lemma}
\begin{proof}
First we assume that $(CD)$ intersects the rhombus $R_{(AB)}$. By
Lemma \ref{easylemma}, neither $C$ or $D$ belongs to the closure
of $R_{(AB)}$. Hence, $(CD)$ intersects the boundary of the rhombus $R_{(AB)}$
at two points, say, $E$ and $F$. Next, since $[EF]\subset(CD)$ does not intersect
$(AB)$, we conclude that the angle $\alpha$ between $[EF]$ and $(AB)$
satisfies the inequality $\tan\alpha<\epsilon$ as required. Similarly,
we prove our assertion if $(AB)$ intersects $R_{(CD)}$.

Now we consider the case where $(CD)$ does not intersect $R_{(AB)}$
and $(AB)$ does not intersect $R_{(CD)}$. Then the boundaries of the
rhombi $R_{(AB)}$ and $R_{(CD)}$ have two common points, say, $E$ and $F$.
The segment $[EF]$ does not intersect the edges $(AB)$ and $(CD)$. Therefore,
the angle $\alpha_1$ between $[EF]$ and $(AB)$ and
the angle $\alpha_2$ between $[EF]$ and $(CD)$ satisfy the inequalities
$\tan\alpha_1<\epsilon$ and $\tan\alpha_2<\epsilon$. Let $\alpha$
be the angle between $(AB)$ and $(CD)$. Then we have
$\alpha\le\alpha_1+\alpha_2$, and the assertion of the lemma follows.
\end{proof}

Finally, to refute the assumption $R_{(AB)}\cap R_{(CD)}\neq \emptyset$,
assume, without loss of generality, that $|AB|=1$,  $|AB|\geq |CD|,$
$A=0$ and $B=1$. Let us now use the conclusion that $(AB)$ and $(CD)$
are close to being parallel, along with Lemma \ref{easylemma} for a
a rough estimate as to where the vertices $C,D$ can be located. They may not
lie inside the open disc with the diameter $(AB)$. Since
$R_{(AB)}\cap R_{(CD)}\neq \emptyset$, $|CD|\leq |AB|$ and
$\tan\alpha \leq \frac{2\epsilon}{1-\epsilon^2}$, where $\alpha$ is the angle between $(AB)$ and $(CD)$,
neither $C$, nor $D$ may possess the
imaginary part, whose absolute value is in excess of $4\epsilon$. If $\epsilon$
is small enough, the real part of the leftmost points, where horizontal lines with
$|\Im z|=4\epsilon$ intersect the circle with the diameter $|AB|=1$ is
$O(\epsilon^2)$. Hence, since $|CD|\leq |AB|$, we arrive in an ample conclusion
that one of the endpoints of $(CD)$, say $C$, must lie inside the open square
box $\{\max(|\Re z|,|\Im z|)<4\epsilon\}$ around $A$,
and $D$ inside the same box translated by $1$, so its centre is now $B$.

This, once again, implies contradiction with the minimality of $T$. Indeed,
now if $\epsilon$ is small enough, the edge $(AB)$, whose length is $1$, can
be deleted in $T$ and replaced by a shorter edge $(AC)$ or $(BD)$, without
violating connectivity or creating  loops. As in the proof of Lemma \ref{easylemma},
the replacement will be $(AC)$ if the unique path from $A$ to $C$ in $T$ incorporates $(AB)$, and $(BD)$ otherwise.

We have exhausted all the possibilities for the mutual alignment of a pair
of edges $(AB)$ and $(CD)$. Thus for
two distinct edges $e_1,e_2\in T$, the open
rhombi $R_{e_1}$ and $R_{e_2}$ are disjoint.
This completes the proof of the claim. $\Box$

\section{Proof of Theorem \ref{diffest}}

\subsection{Lemmata}
The main tool to prove Theorem \ref{diffest} is the Szemer\'edi-Trotter incidence theorem. For any set $\mathcal P$ of points and any set of $\mathcal L$ straight lines in a plane let
$$
I(\mathcal P, \mathcal L)=\{(p,l)\in \mathcal P\times \mathcal L:\;p\in l\}
$$
be the set of incidences.

\begin{theorem}[Szemer\'edi and Trotter \cite{ST}] The maximum number of incidences in $\R^2$ is bounded as follows:
\begin{equation}\label{STe}
|I(\mathcal P,\mathcal  L)|\ll (|\mathcal  P||\mathcal  L|)^{\frac{2}{3}} + |\mathcal P| + |\mathcal L|.
\end{equation}
As a result, if $\mathcal P_t$ (or $\mathcal L_t$) denote the sets of points (or lines) incident to at least $t\geq 1$ lines (or points) of $\mathcal L$ (or $\mathcal P$), then
\begin{equation}\label{work}
\begin{aligned}
|\mathcal P_t|&\ll & \frac{|\mathcal L|^2}{t^3} + \frac{|\mathcal L|}{t},\\
|\mathcal L_t|&\ll & \frac{|\mathcal P|^2}{t^3} + \frac{|\mathcal P|}{t}.
\end{aligned}
\end{equation}
\end{theorem}
Let us note that the linear in $|\mathcal P|,|\mathcal L|$ terms in the
estimates (\ref{STe}, \ref{work}) are essentially trivial and usually of no
interest in the sense of being dominated by the non-linear ones, whenever
these estimates are being used. This is also the case in this paper.

The Szemer\'edi-Trotter theorem is also true in full generality in the plane
over $\C$. This was proved by T\'oth \cite{T}. A more  modern proof came out in a recent paper of Zahl \cite{Z}.  In a particular case, where the point set is a Cartesian product, Solymosi (\cite{So}, Lemma 1) observed that the proof of the $\C^2$
version of the Szemer\'edi-Trotter theorem is considerably more
straightforward than dealing with arbitrary points set in $\C^2$. Although
the geometric part of the forthcoming proof closely follows the construction
in \cite{So}, the point sets to which
we apply the theorem are not necessarily Cartesian products, so strictly
speaking we are using here the general version of the Szemer\'edi-Trotter
theorem in $\C^2$ of T\'oth and Zahl. The estimates (\ref{work}) will be
further used in the $\C^2$ setting without additional comments.

One can easily develop a weighted version of the estimates of the Szemer\'edi-Trotter
theorem, quoted next (see Iosevich et al. \cite{IKRT}). Suppose that each line $l\in \mathcal L$ has been assigned a weight $m(l)\geq 1$. The number of weighted incidences
$i_m(\mathcal P,\mathcal L)$ is obtained by summing over the set
$I(\mathcal P,\mathcal L)$, each pair $(p,l)\in I(\mathcal P,\mathcal L)$ being counted $m(l)$ times.
Suppose that the total weight of all lines is $W$ and the maximum weight per line
is $\mu>0$.
\begin{theorem} \label{STw} The maximum number of weighted incidences between a point set $\mathcal P$ and a set of lines $\mathcal L$,  with the total weight $W$ and  maximum weight per line $\mu$ is bounded as follows:
\begin{equation}\label{STew}
i_m(\mathcal P,\mathcal L)\ll \mu^{\frac{1}{3}} (|\mathcal P|W)^{\frac{2}{3}} + \mu |\mathcal P| + W.
\end{equation}
\end{theorem}

\medskip
The second main ingredient to prove Theorem \ref{diffest} comes from a purely additive-combinatorial observation by Shkredov and Schoen (\cite{SS0}, Lemma 3.1), which has recently allowed for a several incremental
improvements towards a number of open questions in field combinatorics in \cite{SS0},  \cite{SS}, \cite{SV}.

This observation is the content of the following Lemma \ref{twothr}, quoting
which requires some notation also used in the sequel. Through the rest of
this section $A,B$ denote any sets in an Abelian group $(G,+)$. In the
context of the field $\C$, Lemma \ref{twothr} will apply to the addition
operation, so the following notation $E$ will stand for the additive energy,
rather that the multiplicative energy $E_*$, which has been used in
the proof of the sum-product estimate in Theorem \ref{sumest}.

For any $d\in A-A$, set
\begin{equation}
A_d=\{a\in A: \,a+d\in A\}.\label{ad}
\end{equation}
The quantity
$$
E(A,B)=|\{(a_1,a_2,b_1,b_2)\in A\times A\times B\times B:\;a_1-a_2 = b_1-b_2\}|,
$$
is referred to as the additive energy of $A,B$.
By the Cauchy-Schwarz inequality, rearranging the terms in the above definition of $E(A,B)$, one has
\begin{equation}\label{esc}
E(A,B)|A\pm B|\geq |A|^2|B|^2.\end{equation}
Indeed, if $d$ or $x$ is, respectively, an element of $A - B$ or $A + B$ and $n(d)$ or $n(x)$ its the number of its realisations  as a difference or sum of a pair of elements from $A\times B$,  i.e, e.g.
$$
n(d) = |\{(a,b)\in A\times B:\,d=a-b\}|,
$$
the estimate (\ref{esc}) follows from the fact that
\begin{equation}
E(A,B) = \sum_{d\in A-B} n^2(d)=\sum_{x\in A+B} n^2(x).
\label{epm}\end{equation}
The quantity $E(A,A)=E(A)$ is referred to as the (additive) energy of $A$.
Note that according to  (\ref{ad}), $n(d)=|A_d|$, for $d\in A-A$.

We will also need the ``cubic energy'' of $A$, defined as follows:
\begin{equation}\label{threee}
E_3(A) = |\{(a_1,\ldots,a_6)\in A\times\ldots\times A:\;a_1-a_2=a_3-a_4=a_5-a_6\}|.
\end{equation}
This definition implies (see \cite{SS}, Lemma 2) that
\begin{equation}
\label{twothree} E_3(A) = \sum_{d\in A-A} E(A,A_d).
\end{equation}

To see this, let us write for all $d\in A-A$ all quadruples satisfying
\begin{equation}\label{firsteq}
a_1-a_3=a_2-a_4
\end{equation}
with $a_1,a_2\in A, a_3,a_4\in A_d$. The list will contain $\sum_{d\in A-A} E(A,A_d)$
quadruples. Any quadruple is repeated as many times as many different values of
$d$ with $a_3,a_4\in A_d$ exist. This number is the number of pairs $(a_5,a_6)$
with $a_5,a_6\in A$ and
\begin{equation}\label{secondeq}
a_5-a_3=a_6-a_4.
\end{equation}
But the number of collections $(a_1,\dots,a_6)$ of elements from $A$ satisfying both
(\ref{firsteq}) and (\ref{secondeq}) is just $E_3(A)$.

The following statement is part of Corollary 3 in \cite{SS}. Since the formulation we use is slightly different from the original one and in an effort to make this paper self-contained, we have chosen to include its proof as well.

\begin{lemma}\label{twothr} Let $A$ be a finite non-empty additive set. For any $D'\subseteq A-A$, one has
\begin{equation}
\label{th} \sum_{d\in D'} |A_d||A- A_d| \geq \frac{|A|^2\left( \sum_{d\in D'}  |A_d|^{\frac{3}{2}}\right)^2 }{E_3(A)}.
\end{equation}
\end{lemma}

\begin{proof} To verify (\ref{th}) observe that  by the inequality (\ref{esc})
applied to the sets $A,A_d$ for a fixed $d$ we have
$$ \sqrt{|A- A_d|} \sqrt{E(A,A_d)} \geq |A||A_d|.
$$
Multiplying both sides by $\sqrt{|A_d|}$ and summing over $d\in D'$, then applying once again the Cauchy-Schwartz inequality to the left-hand side yields
$$
\sqrt{ \sum_{d\in D'} |A_d||A- A_d| } \sqrt{ \sum_{d\in D'} E(A,A_d)}  \geq |A|\sum_{d\in D'}|A_d|^{\frac{3}{2}}.
$$
Squaring both sides and using (\ref{twothree}) completes the proof of Lemma \ref{twothr}.\end{proof}

\begin{corollary}\label{twoen} Let $A$ be a finite non-empty additive set. For any $D'\subseteq A-A$, one has
\begin{equation}
E(A,A-A)E_3(A) \geq |A|^2\left( \sum_{d\in D'}  |A_d|^{\frac{3}{2}}\right)^2
\label{lbem}\end{equation}\end{corollary}

\begin{proof} The proof is based on an observation that \cite{SS} credits to Katz and Koester (see \cite{KK}) that the left-hand side of (\ref{th}) provides a lower bound for $E(A,A-A)$. Indeed,
each $d \in A-A$ has $|A_d|$ representations $d=u-v$
with $u\in A, v\in A_d$. The same $d$ also has at least $|A-A_d|$ representations
$d=u-v$ with $u,v\in A-A$. Indeed, given $d$, for any $v\in A_d$ and
$a\in A$ one can find $u\in A$ so that $d =  (u-a) - (v-a)$, with $|A-A_d|$ distinct values for the second bracket. Hence, if $n(d)$ is the number of representations of $d$ as an element of $A-A$ and $n'(d)$ -- as an element of $(A-A) - (A-A)$, then
$$
E(A,A-A) = \sum_d n(d) n'(d) \geq \sum_d |A_d||A- A_d|.
$$
This, together with (\ref{th}) completes the proof of Corollary \ref{twoen}. \end{proof}

\begin{remark} In the forthcoming main body of the proof of Theorem \ref{diffest} we will use the Szemer\'edi-Trotter theorem to yield upper bounds for the two energy terms in the left-hand side of the estimate (\ref{lbem}) for the additive point set
$P\subset\C^2$, defined in Section \ref{prelim} as to the ratio and product set
cases.\end{remark}

From now on, the above $D'\subseteq A-A$  be a popular subset of the difference set $A-A$, defined as follows:
\begin{equation}\label{dprime}
D'=\left\{d\in A-A:\; |A_d|\geq \frac{1}{2}\frac{|A|^2}{|A-A|}\right\}.
\end{equation}
Then, since ${\displaystyle \sum_{d\in (D')^c}|A_d| \leq \frac{1}{2}|A|^2}$,
$$
\sum_{d\in D'}  |A_d|^{\frac{3}{2}}\geq\left(\frac{|A|^2}{2|A-A|}\right)^{\frac{1}{2}}\sum_{d\in D'}  |A_d|\geq  \frac{1}{4}\left(\frac{|A|^2}{|A-A|}\right)^{\frac{1}{2}} |A|^2.
$$
Substituting this in the statement of  Corollary \ref{twoen}, let us formulate the result as the final corollary, which summarises the above-mentioned arithmetic component of the argument.
\begin{corollary}\label{arco}Let $A$ be a finite non-empty additive set. Then \begin{equation}
\label{thth}
 E_3(A)E(A,A-A) \gg\frac{|A|^8 }{|A-A|}.
\end{equation}
\end{corollary}

We conclude this preliminary section with a remark discussing some recent
applications of Lemma \ref{twothr} and its corollaries. The content of the
remark is not used directly in the main body of the proof of Theorem
\ref{diffest}.  \begin{remark} The estimate
(\ref{thth}) enabled Schoen and Shkredov \cite{SS}, to achieve progress on
the sum set of a convex set problem. They proved that if $A=f([1,\ldots,N])$, where $f$ is
a strictly convex real-valued function, then $|A-A|\gg|A|^{\frac{8}{5}}\log^{-\frac{2}{5}}|A|$, having improved the previously known exponent $\frac{3}{2}$.
The conjectured exponent in the convex set sum set problem is $2$, modulo a factor of $\log|A|$.
Li \cite{Li} -- see also his recent work with Roche-Newton \cite{LR} --
 pointed out that the approach of \cite{SS} can be adapted to the sum-product problem, using
 a variant of the well-known sum-product construction by Elekes \cite{E}. This improves the exponent $\frac{5}{4}$, obtained by Elekes within his construction to $\frac{14}{11}$, modulo a factor of $\log|A|$.
The same exponent $\frac{14}{11}$, modulo a factor of $\log|A|$, had been coincidentally obtained in Solymosi's work \cite{So1}, as stated in (\ref{wr1}) above. Also recently Jones and Roche-Newton \cite{JRN} applied the estimate (\ref{thth}) to  improve
the best known lower bound
on the size of $A(A+1)$  in the real setting. (The latter paper also contains a new lower bound on $|A(A+1)|$ in a finite field setting, obtained via a different technique.) \end{remark}

\subsection{The main body of the proof of Theorem \ref{diffest}}
Recall the definition of the point set $P$, as well as the quantities $L,N$ in the end of Section \ref{prelim}, relative to either the ratio or product set case. In either case, let us
consider the vector sum set of the set $P\subset \C^2$ with some
point set $Q$, such that $|Q|\geq |P|$. (In the sequel we will set $Q =- P$ or $P- P$). Recall that the set $P$ contains all points of $A\times A$ supported on a popular set of lines through the origin $L$. To obtain the vector sums, one translates the
lines from $L$ to each point of $Q$, getting thereby some set ${\mathcal L}$ of lines with $|{\mathcal L}|\leq |L||Q|$.

In both the ratio and product set cases, it can be assumed that
\begin{equation}
|L|\ge \frac{1}{2}N.
\label{L,N_suppos}\end{equation}
The estimate (\ref{L,N_suppos}) is clear in the ratio set case, where $N\leq |A|\leq 2|L|$.

As to the product set case $|P|\approx |L|N$, we will need the following lemma, which will be used once more in the end of the proof of Theorem \ref{diffest}.

\begin{lemma}\label{ll} There exists $L,N$ satisfying (\ref{min}) and such that
\begin{equation}\label{nbd}
N\ll \frac{|A- A|^2|A\cdot A|}{|A|^3}.\end{equation}\end{lemma}
A variant of Lemma \ref{ll} can be found in the recent papers \cite{LR},
\cite{SS1} and represents a slight generalisation of the well known approach to
the sum-product problem due to Elekes, \cite{E}. The proof of Lemma \ref{ll} is
given in the final section of the paper.

The bound  (\ref{nbd}) for $N$ and the fact that $LN^2$ is bounded from below by (\ref{min}) would yield
under the assumption $|L|\leq N$ that
$$\frac{|A- A|^6|A\cdot A|^3}{|A|^9}\gg LN^2\gg\frac{|A|^4}{|A\cdot A|\log|A|}.$$
Therefore,
$$|A- A|^6|A\cdot A|^4 \gg \frac{|A|^{13}}{\log |A|},$$
which is better than (\ref{resdif}). Thus we assume the estimate (\ref{L,N_suppos}) henceforth.
\medskip

We return to analysing the set of lines $\mathcal L$. The Szemer\'edi-Trotter theorem enables one to estimate $|\mathcal L|$ from below. We have the following estimate for the number of incidences
\begin{equation}
|L||Q|\leq |I( Q,\mathcal L)|\ll |\mathcal L|^{\frac{2}{3}} |Q|^{\frac{2}{3}} + |\mathcal L| + |Q|.
\label{aux}\end{equation}
Since it can be assumed that $|L|$ is bigger than some absolute constant (as the target estimates (\ref{resdif}) are up to absolute constants), the term $|Q|$ in (\ref{aux})
cannot dominate the estimate. Nor can the term $|\mathcal L|$, for otherwise
$|\mathcal L|> |Q|^2$. This, since by construction of $\mathcal L$ one has
$|\mathcal L|\leq  |L||Q|,$ would imply $|L|> |Q|$, but in out set-up
$|Q|\geq|P|\geq|L|.$

Thus it follows from (\ref{aux}) that

\begin{equation}
|{\mathcal L}|\gg |L|^{\frac{3}{2}}|Q|^{\frac{1}{2}}. \label{lowerl}\end{equation} Let us call the number of points of $Q$ on a particular line $l\in \mathcal L$, the weight  $m(l)$ of $l$. The total weight $W$ of all lines in the collection $\mathcal L$
is by construction equal to $|L||Q|$.

Let us now study the set $P+Q$. The vector sums in $P+Q$ are obtained by the parallelogram rule, hence we observe that $P+Q$ is supported on the union of the lines from $\mathcal L$, as subsets of $\mathbb C^2$:
\begin{equation}\label{loc}
P+Q\subset \bigcup_{l\in \mathcal L}l.
\end{equation}

Our goal now is to obtain upper bounds, in terms of $t\geq 1$, on the number of elements of $P+Q$, whose number of realisations as a sum $p+q:\,p\in P,q\in Q$ is at least $t$. The same line
$l\in \mathcal L$ can contribute to the same vector sum $x=p+q\in P+Q, q\in l$, at most
$\min(N,m(l))$ times. In view of this, we can lower the weights of lines, which are
``too heavy'': whenever $m(l)\geq  N$, let us redefine it as $N$.
After this has been done, $W$ denoting the total weight of the lines in ${\mathcal L}$,
one has
\begin{equation}\label{weight}
W\leq |L||Q|.
\end{equation}
Also, we denote
\begin{equation}\label{weight2}
\bar m = \sqrt{\frac{|Q|}{|L|}}.
\end{equation}

The Szemer\'edi-Trotter
theorem, namely (\ref{work}), tells one that the weight distribution over $\mathcal L$  obeys the inverse cube law. I.e.,
for $t\leq N$, one has
\begin{equation}
|\mathcal L_t| =|\{l\in \mathcal L:\, m(l)\geq t\}|\ll \frac{|Q|^2}{t^3}+\frac{|Q|}{t}\ll \frac{|Q|^2}{t^3},
\label{incube}\end{equation}
 as since $N\leq \sqrt{2|L|N}\leq2\sqrt{|P|}\leq 2\sqrt{|Q|}$, the trivial
term $\frac{|Q|}{t}$ gets dominated by the first term. It also follows from
(\ref{incube}), via the standard dyadic summation in $t$, that the total weight
$W(\mathcal L_t)$ supported on the lines from $\mathcal L_t
$ is bounded by
 \begin{equation}
W(\mathcal L_t) \ll \frac{|Q|^2}{t^2},
\label{incubew}\end{equation}
(To see this one partitions $\mathcal L_t$ into ``dyadic subsets'' of lines whose weights $\tau\leq N$ are $2^j t\leq\tau< 2^{j+1} t$ for $j\geq 0$. It follows from (\ref{incube}) that $W(\mathcal L_t) \ll \frac{|Q|^2}{t^2} \sum_{j\geq 0} 2^{-2j+1}.$)

Suppose, in view of (\ref{loc}) some $x\in P+Q$ is incident to $k\geq 1$ lines $l_1,\ldots, l_k\in \mathcal L$. We then have an inequality
\begin{equation} n(x)\leq m(x),\label{nmb}\end{equation}
where
\begin{equation}\label{mx}
n(x)=|\{(p,q)\in P\times Q: x=p+q\}|, \qquad  m(x) = \sum_{i=1}^k m(l_k).
\end{equation}

Observe now that if $n(x)>N$, then $x\in P+Q$ must be incident to more than one line from $\mathcal L$. Indeed, each line $l\in \mathcal L$ may contribute at most $N$ to the quantity $n(x)$.

Hence, let ${\mathcal P}(\mathcal L)$ denote the set of all pair-wise intersections of lines from $\mathcal L$.
We can therefore bound the maximum number of points in $P+Q$, whose number of realisations $n(x)$
is at least $t> N$, in terms of $t$, by way of bounding the number of
$x\in {\mathcal P}(\mathcal L)$, with $m(x)\geq t$. The latter bound will follow from Theorem \ref{STw}
together with the inverse cube weight distribution bounds (\ref{incube},
\ref{incubew}) over the set of lines $\mathcal L$. Namely, we have the following lemma, which also has its prototype in \cite{IKRT}, Lemma 6.

\begin{lemma}\label{wst}
Suppose that $|Q|\geq |P|.$
Then for some absolute $C$ and $t:\;CN\leq t \leq |P|$,
\begin{equation}\label{use}
|\{x\in P+Q:\;n(x)\geq t\}|\ll \frac{|L|^{\frac{3}{2}} |Q|^{\frac{5}{2}}}{t^3},
\end{equation}
\end{lemma}

\begin{proof}
Observe that for any point set ${\mathcal P}$, the number of weighted incidences $i_m (\mathcal P,\mathcal L)$ of ${\mathcal L}$ with ${\mathcal P}$ can be bounded from above using dyadic decomposition of ${\mathcal L}$ by weight in excess of $\bar m$, as
 follows:
\begin{equation}
i_m (\mathcal P,\mathcal L) \leq \sum_{j= 0}^{\lceil\log_2 N/\bar m\rceil}\, i_m(\mathcal P,\mathcal L_{2^j \bar m}).
\label{alt}\end{equation}
Above, the notation $\mathcal L_{\bar m},$  corresponding to $j=0$ stands for the subset of ${\mathcal L}$ containing all
those lines whose weight does not exceed $\bar m$, and
$$\mathcal L_{2^j \bar m}=\{l\in \mathcal L: \;2^{j-1}\bar m < m(l) \leq 2^{j}\bar m\},\qquad j\geq 1.
$$
To estimate each individual term $i_m(\mathcal P, \mathcal L_{2^j\bar m})$ in the sum (\ref{alt}),
one can use the estimate  (\ref{STew}) of Theorem \ref{STw}. The quantity $2^j\bar m$ then replaces the
maximum weight $\mu$ in (\ref{STew}). The total weight $W$ in (\ref{STew}) will be replaced by the total weight $W_{2^j\bar m}$ of the line set $\mathcal L_{2^j\bar m}$.
In view of (\ref{incubew}), the quantity $W_{2^j\bar m}$ is bounded as follows:
\begin{equation}W_{2^j\bar m}\ll \frac{|Q|^2}{2^{2j}\bar m^2}=\frac{|Q||L|}{2^{2j}}.
\label{onewb}\end{equation}

Thus
\begin{equation}\label{estss}
i_m(\mathcal P, \mathcal L_{2^j\bar m}) \ll (2^j\bar m)^{\frac{1}{3}}(|\mathcal P|Q||L|2^{-2j})^{\frac{2}{3}}
+ 2^j\bar m |\mathcal P|+ W_{2^j\bar m}.
\end{equation}

Using (\ref{onewb}), it follows that in the summation (\ref{alt}), the term
$j=0$ dominates the net contribution of the first and the third terms in
the estimate (\ref{estss}) for $j>0$. Conversely, the dominant value of the
second term
in (\ref{estss}) corresponds to the maximum value $N$ of the lines' weight. Thus
\begin{equation}\begin{aligned}
i_m (\mathcal P,\mathcal L) & \ll  \bar m^{\frac{1}{3}}(|\mathcal P|W)^{\frac{2}{3}}  + N{|\mathcal P|} + W \\
&\ll |\mathcal P|^{\frac{2}{3}} |L|^{\frac{1}{2}}|Q|^{\frac{5}{6}} + N{|\mathcal P|} +|L||Q|,\end{aligned}
\label{int}\end{equation}
using (\ref{weight}).

Recall that in view of (\ref{nmb}), for $t>N$, we have the inclusion
\begin{equation}
\{x\in P+Q:\;n(x)\geq t\}\subseteq \left(\mathcal P_t \equiv \{p\in \mathcal P(\mathcal L):\;m(p)\geq t\}\right).\label{pt}\end{equation}
Hence, we apply the incidence bound (\ref{int}) to the point set $\mathcal P_t$, together
with the lower bound
\begin{equation}\label{trlb} t|\mathcal P_t|\leq i_m(\mathcal P_t,\mathcal L).\end{equation}

It follows that for $t\geq CN$, where the constant $C$ is determined by the constants hidden in the $\ll$ symbol in the estimate (\ref{int}), the second term in the right-hand side of the estimate (\ref{int}), applied to the set $\mathcal P_t$, cannot pos
sibly dominate the estimate. Thus, for $t\geq CN\gg N$, one has
\begin{equation}\label{useprime}
|\mathcal P_t|\ll \frac{|L|^{\frac{3}{2}} |Q|^{\frac{5}{2}}}{t^3} + \frac{|L||Q|}{t}.
\end{equation}
It follows that
\begin{equation}\label{useprime1}
|\mathcal P_t|\ll \frac{|L|^{\frac{3}{2}} |Q|^{\frac{5}{2}}}{t^3}, \qquad \mbox{for} \qquad CN\leq t\leq \sqrt[4]{|L||Q|^3}.
\end{equation}

For larger $t$, one has
to be slightly more careful with the term $\frac{|L||Q|}{t}$ in (\ref{useprime}), which, in fact, can be refined for
\begin{equation}\label{imp}t\geq 2|L|\bar m = 2\sqrt{|Q||L|}\leq 2\sqrt[4]{|L||Q|^3}.\end{equation}

Note that the lines in ${\mathcal L}$ come in $|L|$ possible directions, and
therefore no more than $|L|$ lines can be incident to a single point in
$\mathcal P(\mathcal L)$. Hence, lines from a dyadic set
$\mathcal L_{2^j \bar m}$ cannot contribute but
a small proportion to the total number of weighted incidences supported on the sets $\mathcal P_t$, if $t$ is much greater than  $|L|\cdot (2^j \bar m)$.

More precisely, suppose that $t = |L|\cdot(2^i\bar m),\,i\geq 1$. It follows that for such $t$, the estimate (\ref{alt}) can be restated as to the set $\mathcal P_t$ as follows:
\begin{equation}\frac{1}{2}t |\mathcal P_t| \leq \sum_{j= i}^{\lceil\log_2 N/\bar m\rceil}\, i_m(\mathcal P_t,\mathcal L_{2^j \bar m}).\label{refined}\end{equation}

Indeed, the total contribution of the dyadic sets $\mathcal L_{2^j \bar m}$, to the quantity $m(x)$ for $x\in \mathcal P_t$ and  $j<i$ is at most $2^{i-1}|L|\bar m =\frac{t}{2}.$

We now repeat the argument estimating the right-hand side, which has lead from (\ref{alt}) to (\ref{int}), having in mind that it is only the last term $W$ in the first line of (\ref{int}) that needs to be changed. Namely, $W$ should get replaced by the t
otal weight of the lines, contributing to the right-hand side of (\ref{refined}). These are the lines, whose individual weight is at least $\frac{t}{2|L|}$. Let $W_{\frac{t}{2|L|}}$ denote the total weight supported on these lines.
By (\ref{incubew}) we can estimate $$W_{\frac{t}{2|L|}}\ll \frac{|Q|^2|L|^2}{t^2}.$$

Thus for $t\geq 2|L|\bar m$ the estimate (\ref{useprime}) can be improved as follows:
\begin{equation}
|\mathcal P_t|  \ll
\frac{|L|^{\frac{3}{2}}|Q|^{\frac{5}{2}}}{t^3} +  \frac{ |Q|^2|L|^2}{t^3 }.
\label{tailpups}\end{equation}
Since  $|Q|\geq |P|\geq |L|$, the first term in  (\ref{tailpups}) dominates the estimate, and in view of (\ref{imp}), one has
\begin{equation}\label{useprime2}
|\mathcal P_t|\ll \frac{|L|^{\frac{3}{2}} |Q|^{\frac{5}{2}}}{t^3}, \qquad \mbox{for} \qquad  t\geq \sqrt[4]{|L||Q|^3}.
\end{equation}
The estimates (\ref{useprime1}), (\ref{useprime2}) and the inclusion (\ref{pt}) complete the proof of Lemma \ref{wst}.
\end{proof}

\medskip
All the key ingredients to finish the proof of Theorem \ref{diffest} have been developed. We now use Lemma \ref{wst} to yield an upper bound for the left-hand side in the estimate (\ref{thth}) of Corollary \ref{arco}, applied to the additive set $P$.

Using Lemma \ref{wst} with $Q=- P$, we can bound the quantity $E_3(P)$ as follows:

\begin{equation}\label{e3est}
E_3(P)=\sum_{x\in P-P} n^3(x)\ll  N^2|P|^2 + |L|^{\frac{3}{2}} |P|^{\frac{5}{2}} \sum_{j=0}^{\log |A|} 1.
\end{equation}
Above, the first term deals with the set of all $x\in P-P$, whose number of realisations $n(x)$ is less than the applicability threshold $t=CN$ of Lemma \ref{wst}, with some absolute constant $C$. I.e.,
$$\sum_{x\in P-P:\,n(x)<CN} n^3(x) \ll N^2 \sum_{x\in P-P} n(x) \leq N^2|P|^2.$$

The second term in (\ref{e3est}) results from applying Lemma \ref{wst} to the part of the cubic energy supported on $\{x\in P-P:\,CN\leq  n(x)\leq |P|\} $, using dyadic summation.
Namely for $j\geq 0$, let $X_j= \{x: 2^j CN \leq n(x)< 2^{j+1}CN\leq|P|\}$. Then
$$
\sum_{x\in P-P:\,CN\leq  n(x)\leq |P| } n^3(x) \leq \sum_{j\geq 0}|X_j| \cdot (2^{j+1}CN)^3,
$$
and since $X_j$ is nonempty for $j=O(\log |A|)$ only, the bound (\ref{use}) for $|X_j|$, where one sets $t=2^j CN$, results in the second term in (\ref{e3est}).

In both the ratio and product set cases, by (\ref{L,N_suppos}), $N^2\le 4L^2\le 4\sqrt{|P||L|^3}.$
Thus the second term dominates the estimate (\ref{e3est}), that is

\begin{equation}\label{e3estt}
E_3(P)\ll |L|^{\frac{3}{2}} |P|^{\frac{5}{2}} \log|A|.
\end{equation}

Substituting the estimate (\ref{e3estt}) into (\ref{thth}) yields:
\begin{equation}
E(P, P-P) \gg \frac{|P|^{\frac{11}{2}}}{ |L|^{\frac{3}{2}} |P-P| \log|A|}. \label{lowerbd1}
\end{equation}

Now one can also use Lemma \ref{wst}  with $Q=P- P$ to estimate the quantity $E(P, P- P)$ from above. It follows from (\ref{use}) that for any $t\geq C N$ one has:
\begin{equation}
E(P, P- P) \ll |P||P- P| t + \frac{|L|^{\frac{3}{2}} |P- P|^{\frac{5}{2}}}{t}. \label{upperbd}
\end{equation}
Above, the first term gives a trivial bound for the contribution to $E(P, P- P)$ of all those $x\in P+P-P$ which have fewer than $t$ realisations $n(x)$. The second term uses (\ref{use}) and bounds the contribution to $E(P, P- P)$ of the terms with $t$ or
 more realisations: this contribution is bounded by the dyadic sum
$$
\sum_{j=0}^\infty |\{x\in P+P-P:\,n(x)\geq 2^jt\}| (2^{j+1}t)^2\ll \frac{|L|^{\frac{3}{2}} |P- P|^{\frac{5}{2}}}{t}.
$$
Now we can choose
$$t= C'\frac{|P- P|^{\frac{3}{4}} |L|^{\frac{3}{4}}}{\sqrt{|P|}},$$ where the constant $C'$ is large enough to ensure that $t\geq CN$, the applicability threshold of Lemma \ref{wst}. Such a $C'$ exists, since $|P-P|\geq|P|\geq|L|\geq\frac{N}{2}$.

The above choice of $t$ in (\ref{upperbd}) yields
\begin{equation}
E(P, P- P) \ll \sqrt{|P|}|P- P|^{\frac{7}{4}} |L|^{\frac{3}{4}}. \label{upperbddone}
\end{equation}

Combining this with (\ref{lowerbd1}) results in the following inequality:
\begin{equation}\label{ngood1}
|P-P|^{\frac{11}{4}} |L|^{\frac{9}{4}} \gg \frac{|P|^5}{\log |A|}.
\end{equation}

\medskip
It remains to eliminate $|L|$ from the latter estimate, relative to the ratio or the product set case.

To obtain the first estimate of (\ref{resdif}) as to the ratio set case, it suffices to note that $|P|\geq \frac{1}{2} |A|^2$, $|L|\leq |A:A|, $
as well as $|P-P| \leq |A-A|^2$.

In the product set case, where  $|P|\approx |L|N$, the estimate (\ref{ngood1}) becomes

\begin{equation}\label{ngood3}
|A-A|^{\frac{11}{2}}  \gg \frac{(|L|N^2)^{\frac{11}{4}}}{\sqrt{N} \log |A|}.
\end{equation}
The quantity $LN^2$ is bounded from below by (\ref{min}), and Lemma \ref{ll} provides a non-trivial upper bound (\ref{nbd}) for $N$.
Substituting these bounds into (\ref{ngood3}) yields the second estimate of (\ref{resdif}) and completes the proof of Theorem \ref{diffest}. \qed

\subsection{Proof of Lemma \ref{ll}}

A variant of Lemma \ref{ll} can be found in the recent papers  \cite{LR}, \cite{SS1} and represents a slight generalisation of the well known approach to the sum-product problem
due to Elekes \cite{E}.
For completeness sake, we further present a simple proof.
The notation in the forthcoming argument is somewhat independent from the rest of the paper.

Consider a set $A$, not containing zero and a set of lines
$\mathcal L = \{y = \frac{d + x}{a}\}$, where $d$ is an element of
the difference set $A-A$ and $a\in A$. Clearly there are $|A-A||A|$ lines.
Therefore, the number of points in a set $\mathcal P_t$, where more than $t$ lines from $\mathcal L$ intersect is, by
 (\ref{work}), bounded as follows:
 \begin{equation}\label{ptbd}
 |\mathcal P_t|\ll \frac{|A-A|^2|A|^2}{t^3} + \frac{|A-A||A|}{t}.
 \end{equation}

Suppose now that
$$L_{t}=\{l\in A:A, \, n(l)> t\}.$$
For each
$l\in L_{t}$, one has $l=\frac{a'_i}{a_i}$, where the index $i$ runs over $n(l)$
distinct values. Given $l\in L_{t}$, for
every  $a\in A$, one has $l=\frac{(a'_i - a) +a }{a_i}$, for $i=1,\ldots,n(l)$. I.e., the point in the plane with coordinates $(a,l)$ is incident to at least $n(l)$ lines from  $\mathcal L$, these lines being identified by the pairs $(d_i = a_i'-a, a_i)
$, with $i=1,\ldots,n(l)$.

 Hence $A\times L_{t} \subseteq \mathcal P_{t}$, and
 it follows from (\ref{ptbd}) that
\begin{equation}\label{livar}
 |L_{t}|\ll \frac{|A-A|^2|A|}{t^3}+\frac{|A-A|}{t}.
 \end{equation}

 Let us use (\ref{livar}) to estimate the contribution of the set $L_t\subseteq A:A$ to the multiplicative energy $E_*(A).$
 For $j=0,1,\ldots$, with the upper bound $2^{j+1}t\leq|A|$, the set of ratios $\{l\in A:A, \, 2^{j}t< n(l)\leq 2^{j+1}t\leq|A|\}$ contributes to $E_*(A)$ at most
 $$
 4 |L_{2^{j}t}| (2^{j}t)^2 \ll \frac{|A-A|^2|A|}{2^{j}t}+ |A-A|(2^{j}t).
 $$
Summing the right-hand side over $j$ yields a bound for the contribution of the set $L_t$ to the multiplicative energy $E_*(A)$, as follows:
$$
\sum_{l\in A:A, \,t<n(l)\leq|A|} n^2(l)\ll \frac{|A-A|^2|A|}{t} + |A-A||A| \ll \frac{|A-A|^2|A|}{t}.
$$
Comparing this with the lower bound (\ref{emult}) for $E_*(A)$ shows that for some $C$, one can set
\begin{equation}
\label{tbnd}
t = C\frac{|A\cdot A||A-A|^2}{|A|^3},\end{equation}
and have the following inequality

$$
\sum_{l\in A:A, \,n(l)\leq t} n^2(l) \geq\frac{1}{2} \frac{|A|^4}{|A\cdot A|}.$$

Thus, there exists a dyadic subset of $\{l\in A:A, \,n(l)\leq t\}$, namely
the set $L=\{l\in A:A, \,\frac{N}{2}<n(l)\leq N\},$ for some $N\leq t$, such that this set $L$ contributes to the multiplicative energy $E_*(A)$ at least the amount $\frac{1}{2\log_2|A|} \frac{|A|^4}{|A\cdot A|}.$
Since $t$ satisfies (\ref{tbnd}), this proves Lemma \ref{ll}.
 \qed

\begin{remark} The argument in the above proof of Lemma \ref{ll} is symmetric with respect to the
two field operations in $\C$:
by defining the set of lines as $\mathcal L = \{y = lx -a\},$ where
$(l,a)\in (A:A)\times A$ one can get a similar upper bound on the maximum number of realisations of popular differences (or sums, by a trivial modification), contributing to the additive energy $E(A)$, via the ratio or product set.  \end{remark}

\section{Acknowledgement} The first author was partially supported
by Russian Fund for Basic Research, Grant N.~11-01-00329,
and Program Supporting Leading Scientific Schools, Grant Nsh-6003.2012.1.
The second author thanks Oliver Roche-Newton for helpful discussions and remarks.


\newpage
\begin{thebibliography}{4}
\bibitem{BJ} T. Bloom, T. Jones. {\em A sum-product theorem in function fields.} Preprint {\sf arXiv math:1211.5493} (2012), 17pp.

\bibitem{E} G. Elekes. {\em On the number of sums and products}. Acta Arithmetics {\bf 81} (1997), 365--367.

\bibitem{ENR} G. Elekes, M. Nathanson, I. Z. Ruzsa. {\em  Convexity and sumsets.} J. Number Theory {\bf 83} (2000), 194--201.

\bibitem{ES} P. Erd\H os, E. Szemer\'edi. {\em On sums and products of integers.} Studies in Pure Math. (Birkh\"auser, Basel, 1983) 213--218.

\bibitem{G} M.Z. Garaev. {\em On an additive representation associated with the $L^1$-norm of an exponential sum.} Rocky Mountain J. Math. {\bf 37} (2007), no. 5, 1551--1556.

\bibitem{IKRT} A. Iosevich, S. Konyagin, M.  Rudnev, V Ten. {\em Combinatorial complexity of convex sequences.} Discrete Comput. Geom. {\bf 35} (2006), no. 1, 143--158.

\bibitem{JRN} T. Jones, O. Roche-Newton. {\em Improved bounds on the set $A(A+1).$} Preprint {\sf arXiv math:1205.3937} (2012), 15pp. To appear in J. Combin. Th. A.


\bibitem{KK} N.H. Katz, P. Koester. {\em On additive doubling and energy.} SIAM J. Discrete Math., 24(4) (2010), 1684--1693.

\bibitem{Li} L. Li.  {\em On a theorem of Schoen and Shkredov on sumsets of convex sets.} Preprint {\sf arXiv math:11108.4382} (2011), 6pp.

\bibitem{LR} L. Li, O. Roche-Newton. {\em Convexity and a sum-product type estimate.} Preprint {\sf arXiv: math: 1111.5159} (2011), 10pp.

\bibitem{MR} M. Rudnev. {\em An Improved Sum-Product Inequality in Fields of Prime Order.} Int. Math. Res. Notices (2012) (16): 3693--3705.

\bibitem{SS0} T. Schoen, I. Shkredov. {\em  Additive properties of multiplicative subgroups of $\F_p$.}  Quart. J. Math. {\bf 63} (2012), 713--722

\bibitem{SS} T. Schoen, I. Shkredov. {\em On sumsets of convex sets.} Comb. Probab. Comput. {\bf 20}
(2011), 793--798.

\bibitem{SS1} T. Schoen, I. Shkredov.  {\em Higher moments of convolutions.}  Preprint {\sf arXiv math:1110.2986} (2011), 36pp.

\bibitem{SV} I. Shkredov, I. Vyugin. {\em  On additive shifts of multiplicative subgroups.} Sb. Math., 2012, {\bf 203} (6), 844–-863.

\bibitem{So} J. Solymosi. {\em On the number of sums and products} Bull. London Math. Soc. {\bf 37} (2005), no. 4, 491--494.

 \bibitem{So1} J. Solymosi.   {\em Bounding multiplicative energy by the sumset.} Adv. Math. {\bf 222} (2009), no. 2, 402--408.

 \bibitem{SoT} J. Solymosi, T. Tao. {\em An incidence theorem in higher dimensions.} Preprint {\sf arXiv math: 1103.2926} (2011), 24pp.

\bibitem{ST} E. Szemer\'edi, W.T. Trotter, Jr. {\em Extremal problems in discrete geometry.} Combinatorica {\bf 3} (1983), 381--392.

\bibitem{TT} T. Tao. {\em The sum-product phenomenon in arbitrary rings.} Cont. to Disc.
Math. {\bf 4} (2009), 59--82.

 \bibitem{T}  C. D. T\'oth. {\em   The Szemer\'edi-Trotter Theorem in the Complex Plane.}
 Preprint {\sf arXiv math/0305283} (2003), 23pp.

 \bibitem{Z} J. Zahl. {\em A Szemer\'edi-Trotter type theorem in $\R^4$.} Preprint {\sf arXiv:math/1203.4600} (2012), 47pp.

\end{thebibliography}
\end{document}